\documentclass[12pt]{article}
\usepackage{amssymb}
\usepackage[dvips]{graphicx}
\usepackage{latexsym}
\usepackage{amsmath}
\def\qed{\hfill$\Box$}

\newtheorem{theorem}{Theorem}
\newtheorem{proposition}{Proposition}

\newtheorem{lemma}{Lemma}

\newtheorem{definition}{Definition}

\newcommand{\Z}{\mbox{${\mathbb
 Z}$}}

\title{Null and non--rainbow colorings of projective plane and sphere triangulations}

\author{Jorge L. Arocha and Amanda Montejano}

\begin{document}

\maketitle

\begin{abstract}
For maximal planar graphs of order $n\geq 4$, we prove that a vertex--coloring containing no rainbow faces uses at most $\left\lfloor\frac{2n-1}{3}\right\rfloor$ colors, and this is best possible. For  maximal graph embedded on the projective plane, we obtain the analogous best bound $\left\lfloor\frac{2n+1}{3}\right\rfloor$. The main ingredients in the proofs are classical homological tools. By considering graphs as  topological spaces, we introduce the notion of a null coloring,  and prove that for any graph $G$ a maximal null coloring $f$ is such that the quotient graph $G/f$ is a forest. 
\end{abstract}

\section{Introduction}
For maximal planar graphs embedded on the sphere, or maximal graphs embedded  on the projective plane, we study vertex colorings containing no tricolored faces.  We study the exact value from which tricolored faces are inevitable. Giving a graph embedded on a surface, the problem of maximizing the number of colors in a vertex coloring  avoiding rainbow faces, has been studied recently in several papers. Before summarizing some known results, we give few definitions. 

Let $G$ be a graph embedded on a surface. A \emph{$k$--coloring} of $G$ is a surjective mapping $f:V(G)\to \{1,2,...k\}$ or, equivalently, a partition of $V(G)$ into exactly $k$ parts called \emph{color classes}. A face of $G$ is said to be \emph{rainbow} if its vertices receive mutually distinct colors. A coloring which contains no rainbow faces is called a \emph{non--rainbow} coloring. The largest integer for which there is a non--rainbow coloring of $G$ is denoted by $\chi_f(G)$. We shall observe that there exist a non--rainbow $k$--coloring of $G$ for every $k\leq \chi_f(G)$, while all $k$--colorings of $G$ with $k> \chi_f(G)$ contain a rainbow face. Thus $\chi_f(G)+1$ is the minimum integer for which every coloring of $G$ contains a rainbow face (in this sense the problem is considered an extremal anti--Ramsey problem).

Most of the research concerning this subject has been done in the context of planar graphs, that is graphs embedded in the sphere. Ramamurthi and West \cite{rw} studied tight lower bonds for $\chi_f(G)$ where $G$ is a $2$--, $3$--, or $4$--chromatic planar graph.  Jendrol' and Schr\"{o}tter \cite{js} determine $\chi_f(G)$ for all semiregular polyhedra. Jungi\'c, Kr\'{a}l'and Skrekovski \cite{jks} investigate the problem for triangle--free planar graphs and provided sharp lower bounds in terms of the girth $g$, for every $g\geq 4$. Dvo\v{r}\'{a}k, Kr\'{a}l' and \v{S}krekovski\cite{dks} study upper bounds for $3$--, $4$-- and $5$--connected planar graphs. In particular, for  a $3$--connected planar graph $G$, they obtain the best possible upper bound $\chi_f(G)\leq \left\lfloor\frac{7n-8}{9}\right\rfloor$. Related works concerning graphs embedding in other surfaces are~\cite{neg,abn1,abn2}.

In this paper we investigate the problem  for maximal planar graphs (also called \emph{sphere triangulations}), and maximal graphs embedded in the projective plane (also called \emph{projective plane triangulations}). In both cases we provide a sharp upper bound slightly  smaller than the upper bound mentioned in the previous paragraph for $3$--connected planar graphs.

\begin{theorem}\label{thm:ub_s}
A non--rainbow coloring of a maximal planar graph with $n\geq 4$ vertices uses at most $\left\lfloor\frac{2n-1}{3}\right\rfloor$ colors.
A non--rainbow coloring of a maximal graph embedded on a projective plane with $n\geq 4$ vertices uses at most $\left\lfloor\frac{2n+1}{3}\right\rfloor$ colors.
\end{theorem}

In order to prove Theorem~\ref{thm:ub_s}, we use classical topological ideas. By considering a graph $G$ as a topological space, and a coloring $f$ as a continuous mapping $f:G\to K_k$, we will introduce the notion of a \emph{null coloring} at the level of homology (see precise definitions in Section~\ref{sec:null}). Our main result, Theorem~\ref{main}, states that for any graph $G$, a maximal null coloring is such that $G/f$ is a forest (here $G/f$ is the quotient graph of the coloring). We deduce Theorem~\ref{thm:ub_s}  from Theorem~\ref{main} by proving that, for maximal planar graphs, or maximal graphs embedded in the projective plane, a maximal non--rainbow coloring is a maximal null coloring (Lemma~\ref{lem}).

The paper is organized as follows. In section~\ref{sec:null} we define null colorings for graphs and prove Theorem~\ref{main}. In Section~\ref{sec:ub} we prove Theorem~\ref{thm:ub_s}  by means of Theorem~\ref{main}. Finally, in Section~\ref{sec:sharp}, we show thatTheorem~\ref{thm:ub_s} is best possible.

\section{Null colorings }\label{sec:null}

In this section we regard graphs as topological spaces (that is, vertices are actually points, and edges should be thought as lines connecting the points corresponding to its two incident vertices). In this seeing, we denote by $H_1(G)$ the first homology group of a connected graph $G$, which is a free abelian group with $m-n+1$ generators, where $n$ and $m$ are respectively the number of vertices and edges of the graph. 
In fact, a basis $\beta $ of generators for $H_{1}(G)$ is given by choosing a spanning tree $T$ of $G$, and an orientation for the edges $\{e_{1},...e_{q-p+1}\}$ of $G-T$.  So, giving a closed path $P$ in $G$, it represent the element $\alpha _{1}\oplus ...\oplus \alpha _{q-p+1}\in \mathbb{Z}\oplus ...\oplus \mathbb{Z}=H_{1}(G)$, according with the basis $\beta$,  where $\alpha _{i}$ is the directed sum of how many times the directed edge $e_{i}$ is transversed by $P$.

It is well known that a $k$--coloring of a graph $G$ can be view as a homomorphism $f:G\to K_k$ (note that since we allow colorings which are not proper, then the homomorphism may be reflexive). By considering graphs as  topological spaces, we can think on a $k$--coloring of $G$ also as a continuous mapping $f:G\to K_k$ (by sending each vertex of $G$ to its image in $K_k$ and extending the map linearly to the edges). Thus, at the level of homology, a $k$--coloring  $f$ induce a group homomorphism $f_*:H_1(G)\to H_1(K_k)$. 

\begin{definition}
For a graph $G$, a $k$--coloring $f:G\to K_k$ is called a \emph{null coloring} if and only if  $f_*:H_1(G)\to H_1(K_k)$ is zero.  
\end{definition}

Let $f$ be a coloring of $G$, we define the \emph{quotient graph} $G/f$ as the graph with vertices the color classes, and two color class are adjacent if there is an edge in $G$ whose incident vertices have those colors. 

In order to prove Theorem~\ref{main} below, we first prove some lemmas. Let $G$ be a graph and $f$ be a coloring of $G$. Let $u,v \in V(G)$ be two vertices with the same color, that is $f(u)=f(v)$. Denote by $G'$ the graph which is obtained from $G$ by identifying the vertices $u$ and $v$, and denote this vertex of $G'$ by $uv$. Let $h: G\to G'$ be the corresponding homomorphism. The coloring $f$ of $G$ naturally induces a coloring $f'$ of $G'$ by defining $f'(uv)=f(u)=f(v)$ and $f'(x)=f(x)$ if $x\notin \{u,v\}$. Obviously $G/f=G'/f'$. We denote by $d(u,v)$ de distance between $u$ and $v$ in $G$.

\begin{lemma}\label{lem:1}
For a coloring $f$ of $G$, and two vertices $u,v\in V(G)$ with the same color, let $G'$, $f'$, and $h:G\to G'$ as defined above. If $d(u,v)\leq 2$, then $h_*:H_1(G)\to H_1(G')$ is an epimorphism. 
\end{lemma}

\begin{proof}
If $u$ and $v$ are adjacent vertices, then $h$, thought as a topological map from $G$ to $G'$, is a homotopy equivalence because it is just the contraction of the edge $uv$, and hence  $h_*:H_1(G)\to H_1(G')$ is an isomorphism. 

Suppose now that $d(u,v)=2$, and let $w$ be a vertex adjacent to both $u$ and $v$ in $G$.  Let  $\gamma$ be a cycle of $G'$, we consider the following cases:

(1) $\gamma$ does not contains $uv$. Then, obviously the corresponding cycle of $G$ is sent, by $h$, to $\gamma$. 

(2) $\gamma=\{uv,z_1,z_2,...,z_{\rho},uv\}$ and either $\{u,z_1,z_2,...,z_{\rho},u\}$ or $\{v,z_1,z_2,...,z_{\rho},v\}$ is a cycle in $G$. Again the corresponding cycle is sent, by $h$, to $\gamma$.

(3) $\gamma=\{uv,z_1,z_2,...,z_{\rho},uv\}$ and $\{u,z_1,z_2,...,z_{\rho},v\}$ is a path in $G$. If $w\neq z_i$ for any $1\leq i\leq \rho$, then consider the cycle $\gamma^0$ of $G$, $\gamma^0=\{u,z_1,z_2,...,z_{\rho},v,w,u\}$. Clearly, at the level of homology, the element of $H_1(G)$ induced by $\gamma^0$ is sent, by $h_*$, to  the element of $H_1(G')$ induced by $\gamma$. Suppose now that $w=z_{i_0}$ for some $1\leq i_0 \leq \rho$. Let $\gamma^1=\{u,z_1,z_2,...,z_{i_0}=w,u\}$ and $\gamma^2=\{v,w=z_{i_0},z_{i_0+1},...,z_{\rho},v\}$. Note that, at the level of homology, the sum of the elements of $H_1(G)$ induced by $\gamma^1$ and $\gamma^2$ is sent, by $h_*$, to   the element of $H_1(G')$ induced by $\gamma$. 

In any case,  $h_*:H_1(G)\to H_1(G')$ is an epimorphism. 
\qed
\end{proof}

\begin{lemma}\label{lem:2}
For a maximal null coloring $f$ of $G$,  and two vertices $u,v\in V(G)$ with the same color, let $G'$, $f'$, and $h:G\to G'$ as defined above. If $d(u,v)\leq 2$,  then $f'$ is a maximal null coloring of  $G'$.

\end{lemma}

\begin{proof}
By  Lemma~\ref{lem:1}, $f'$ is a null coloring of $G'$.  We shall prove that $f'$ is a maximal null coloring. Suppose there exist a null coloring $g'$ of $G'$ using more colors than $f'$. Then we extend $g'$ to a coloring $g$ of $G$, by letting $g(u)=g(v)=g'(uv)$.  Since, at the level of homology,  $g_*=g'_*\circ h_*$, then $g$ is a null coloring of $G$, contradicting the maximality of $f$. 

\qed
\end{proof}

Given a coloring $f$ of a graph $G$ and two different  colors $i,j$, we called an edge $e\in E(G)$ an \emph{$(ij)$--edge} if the vertices of $e$ have colors $i$ and $j$ respectively.

\begin{lemma}\label{lem:3}
Let $f:G\to K_k$ be a null coloring. For any pair of different colors $i,j\in V(K_k)$ and any cycle $C$ of $G$, the number of $(ij)$--edges in $C$ is even.
\end{lemma}

\begin{proof}
Let $T$ be a spanning tree of $K_k$ such that the edge $(i,j)\notin E(T)$, and let $\beta $ be a basis of  $H_{1}(K_k)$ given by some orientation of  $E(K_k)\setminus E(T)$. Since $f$ is a null coloring, then the closed path $f(C)$ most represent the element $0\oplus ...\oplus 0\in \mathbb{Z}\oplus ...\oplus \mathbb{Z}=H_{1}(K_k)$. Hence the number of times that the path $f(C)$ crosses the edge $(i,j)\in E(K_k)$ is even, as any time that $(i,j)$ is taken in one direction it has to be taken in the oder direction at some point in the path.

\qed
\end{proof}

\begin{lemma}\label{lem:4}
Let $f$ be a maximal null coloring of a graph $G$. If any two vertices $u,v\in V(G)$  with the same color satisfy $d(u,v)\geq 2$,  then $G$ is a bipartite graph.
\end{lemma}

\begin{proof}
By assumption there are no monochromatic edges, and by Lemma~\ref{lem:3} the number of dichromatic edges in any cycle is even. Thus, any cycle of $G$ has an even number of vertices.
\qed
\end{proof}

Now, we are ready to prove our main theorem which is the primary tool in proving Theorem~\ref{thm:ub_s}.

\begin{theorem}\label{main}
If $f$ is a maximal null coloring of a graph $G$, then $G/f$ is a forest.
\end{theorem}

\begin{proof}
It suffices to prove the theorem for connected graphs. The proof is by induction on the order of $G$. The theorem is straightforward for graphs with less that four vertices. Suppose there is a minimum $n$ such that there exist a graph $G$ of order $n$ contradicting the theorem.  

Assume first that there are two vertices $u,v\in V(G)$ with $f(u)=f(v)$ and $d(u,v)\leq 2$. By Lemma~\ref{lem:2}, $f'$ is a maximal null coloring of $G'$. Since $G/f=G'/f'$ and $G'$ has $n-1$ vertices, we get a contradiction. 

Assume now that for any pair of vertices $u,v\in V(G)$ with $f(u)=f(v)$ then $d(u,v)>2$. Then $G$  is bipartite according to Lemma~\ref{lem:4}. So, there is an independent set $I\subset V(G)$ such that $|I|\geq \lceil \frac{n}{2}\rceil$, where $|V(G)|=n$. From which we conclude that $f$ uses at least  $ \lceil \frac{n}{2}\rceil + 1$ colors, since we can define a null coloring of $G$ by assigning to each vertex of $I$ a distinct color, and to all vertices in $V(G)\setminus I$ one more color. Therefore, by the pigeonhole principle  there  must be a vertex $w\in V(G)$ such that $f(w)$ is unique. Next, we shall prove that $w$ is either a cut point or a terminal point of $G$. Suppose there is a cycle $\gamma$ of $G$ through $w$, $\gamma=\{w,z_1,z_2,...,z_{\rho},w\}$. Since $f(w)$ is unique,  Lemma~\ref{lem:3} implies that $f(z_1)=f(z_{\rho})$, a contradiction as $d(z_1,z_{\rho})=2$.  Now, let $deg(w)=r$ be the number of neighbors of $w$ in $G$, and let $G_1$, $G_2$,... ,$G_r$ be the connected components of $G-w$. Let $f_1$,...,$f_r$ be the restrictions of $f$ to each $G_1$,... ,$G_r$. It is not difficult to see that every $f_i$ is a maximal null coloring. So, by induction $G_i/f_i$ is a tree for every $1\leq i \leq r$. Note also that each $f_i$ uses district colors (otherwise $f$ would not be maximal). Then we conclude that $G/f$ is a tree.

\qed
\end{proof}

\section{Non--rainbow colorings}\label{sec:ub}

In this section we prove Theorem~\ref{thm:ub_s}  by relating the concepts of null and non--rainbow colorings concerning  maximal graphs embedded on  a sphere or a projective plane.

\begin{lemma}
Let $T$ be a maximal graph embedded on a sphere or a projective plane. Then, a non--rainbow coloring of $T$ is also a null coloring of $T$.
\end{lemma}

\begin{proof}
Let $S$ denote either a sphere or a projective plane. Note that a maximal graph $T$ embedded on $S$ is the $1$--skeleton of a triangulation  of $S$. Let $f$ be a non--rainbow $k$--coloring of $G$. Recall that $f$ may be thought as a continuous mapping from $T$ to $K_k$. The fact that $f$ is a non--rainbow coloring allows us to extend the map $f$, triangle by triangle, to a map $F:S\to K_k$. We shall note that, a the level of homology, $F_*:H_1(S)\to H_1(K_k)$ is zero, since $H_1(S)$ is either zero or $\Z_2$ and $H_1(K_k)$ is free abelian. Furthermore, if we denote by $i:T\to S$ the inclusion, then $F\circ i=f$. So, by functoriality $F_*\circ i_*=f_*$, and therefore $f_*:H_1(T)\to H_1(K_k)$ is zero. Thus, $f$ is a null coloring of $T$ as claimed. 

\qed
\end{proof}

Note that the converse of the previous lemma is obviously true, from which it follows directly the next.

\begin{lemma}\label{lem}
Let $T$ be a maximal graph embedded on a sphere or a projective plane. Then, a maximal non--rainbow coloring of $T$ is also a maximal null coloring of $T$.

\qed
\end{lemma}

Now, by means of Theorem~\ref{main} and Lemma~\ref{lem} we can prove Theorem~\ref{thm:ub_s}.

\vspace{.5cm}

\textbf{Proof of Theorem~\ref{thm:ub_s}:} Let $S$ be a sphere, and let $T$ be a maximal planar graph embedded on $S$ with $|V(T)|=n\geq 4$. Let $\chi_f(T)=k$, and consider $f$ a non--rainbow $k$--coloring  of $T$. By definition $f$ is a maximal non--rainbow coloring of $T$, and thus, by Lemma~\ref{lem}, $f$ is a maximal null coloring of $T$. Then, Theorem~\ref{main} implies that $T/f$ is a forest. 
Since $T$ is a connected graph, then the quotient graph $T/f$ is a tree.  Note that $T/f$ has $k$ vertices, and so it has $k-1$ edges. It is not difficult to see that, for each pair of adjacent color classes  there are at least three bi--chromatic faces using those colors. Thus, the total number of faces in $T$, which is $2n-4$ according to Euler's formula,  most be grater than $3(k-1)$. From this follows that $k=\chi_f(T)\leq  \left\lfloor\frac{2n-1}{3}\right\rfloor$, which concludes the proof of the theorem for maximal planar graphs. 

If $S$ is a projective plane, the proof is completely analogous except that the number of faces in a maximal graph of order $n\geq 4$ embedded on $S$, is $2n-2$. Which leads to the upper bound $k=\chi_f(T)\leq  \left\lfloor\frac{2n+1}{3}\right\rfloor$. \qed

\section{Sharpness }\label{sec:sharp}

To prove the sharpness of Theorem~\ref{thm:ub_s} we need to exhibit, for any $n\geq 4$, a maximal planar graph $T$ (respectively a maximal graph embedded in the projective plane) with $n$ vertices, and a non--rainbow coloring  of $T$ using $\left\lfloor\frac{2n-1}{3}\right\rfloor$  (respectively $\left\lfloor\frac{2n+1}{3}\right\rfloor$) colors. This can be done by taking face subdivisions of suitable given sphere or projective plane triangulations.

Let $S$ denote a surface which is either a sphere or a projective plane, and let $T$ be a maximal planar graph embedded on $S$.  To \emph{subdivide} a face $\{v_1,v_2,v_3\}$ of $T$ is to add a new vertex $u$ and the three edges $uv_1$, $uv_2$ and $uv_3$. This operation results in a new  triangulation of $S$.

\begin{proposition}
For every $n\geq 4$, there is a sphere triangulation with
$n$ vertices that has a non-rainbow coloring with
$\lfloor\frac{2n-1}{3}\rfloor$ colors.
\end{proposition}

\begin{proof}
Let $T_k$ be a sphere triangulation with $k=n-\lfloor\frac{2n-1}{3}\rfloor+1$ vertices. Let $T$ be the sphere
triangulation obtained by subdividing $\lfloor\frac{2n-1}{3}\rfloor-1$ faces of $T_k$. 

Note that $T$ has exactly $n$ vertices. Now, let us give  a non-rainbow coloring of $T$ with $\lfloor\frac{2n-1}{3}\rfloor$ colors. Let the original vertices of $T_k$ be colored by one single color, and each new vertex be colored with a different color
\qed
\end{proof}

\begin{proposition}
For every $n\geq 4$, there is a projective plane triangulation with
$n$ vertices that has a non-rainbow coloring with
$\lfloor\frac{2n+1}{3}\rfloor$ colors.
\end{proposition}

\begin{proof}
Just start whit a projective plane triangulation with $k=n-\lfloor\frac{2n+1}{3}\rfloor+1$ vertices. 
\qed
\end{proof}

\end{document}